\documentclass[]{alggeom}
\usepackage{amsmath}
\usepackage[all]{xy}
\usepackage[french]{babel}

\renewcommand{\contentsname}{Contents\\{\footnotesize\normalfont(A table
of contents should normally not be included)}}

\renewcommand{\qed}{{\hfill$\square$}}

\newtheorem{theo}{Th\'{e}or\`{e}me}[section]
\newtheorem{prop}[theo]{Proposition}

\newtheorem{lem}[theo]{Lemme}
\newtheorem{cor}[theo]{Corollaire}
\theoremstyle{definition}
\newtheorem{defi}[theo]{D\'efinition}
\theoremstyle{remark}
\newtheorem{rema}[theo]{Remarque}

\newcommand{\bthe}{\begin{theo}}
\newcommand{\ble}{\begin{lem}}
\newcommand{\bpr}{\begin{prop}}
\newcommand{\bco}{\begin{cor}}
\newcommand{\bde}{\begin{defi}}
\newcommand{\ethe}{\end{theo}}
\newcommand{\ele}{\end{lem}}
\newcommand{\epr}{\end{prop}}
\newcommand{\eco}{\end{cor}}
\newcommand{\ede}{\end{defi}}

\newcommand{\Z}{{\mathbb Z}}

\newcommand{\C}{{\mathbb C}}

\def\P{{\bf P}}
\numberwithin{equation}{section}

\newfont{\gothic}{eufb10}

\begin{document}

 \title[$CH_{0}$-trivialit\'e   d'hypersurfaces cubiques] {$CH_{0}$-trivialit\'e    universelle d'hypersurfaces cubiques presque diagonales}

\author{Jean-Louis Colliot-Th\'el\`ene}
\email{jlct@math.u-psud.fr}
\address{Laboratoire de Math\'ematiques d'Orsay, Universit\'e Paris-Sud, CNRS, Universit\'e
Paris-Saclay, 91405 Orsay, France}

%
 
%
%\dedication{A dedication can be included here.}
%\url{http://www.ams.org/msc/} for a list of codes.}
\thanks{}
%This file documents \pkg{alggeom} version \Fileversion\ and
%was last revised \Filedate.}

\begin{abstract}
 If a smooth cubic hypersurface of dimension at least 2 over the complex field is defined by the vanishing
 of a sum of cubic forms in independent variables and each of these forms involves at most 
 3 variables, then the cubic hypersurface is universally $CH_{0}$-trivial : there is an
 integral Chow  decomposition of the diagonal.
%\vspace*{-15pt}
\end{abstract}

\maketitle

%\vspace*{6pt}\tableofcontents  
% for this guide only.
% A table of contents should normally not be included

\section*{Introduction}
\label{sec:introduction}

Pour $X$ une vari\'et\'e propre sur un corps  on note
$CH_{0}(X)$ le groupe de Chow des z\'ero-cycles  de $X$ modulo l'\'equivalence rationnelle
et on note $A_{0}(X) \subset CH_{0}(X)$ le sous-groupe des classes de z\'ero-cycles de degr\'e z\'ero.

On dit   \cite{ACTP, CTP} qu'une $\C$-vari\'et\'e projective et lisse $X$ est universellement $CH_{0}$-triviale
si, pour tout corps $F$ contenant $\C$ (et l'on ne consid\'erera d\'esormais que de tels corps),
l'homomorphisme ${\rm deg}_{F} : CH_{0}(X_{F}) \to�\Z$ est un isomorphisme, soit encore
$A_{0}(X_{F})=0$.

Si $X$ est rationnelle, ou stablement rationnelle, ou r\'etractilement rationnelle, alors
$X$ est universellement $CH_{0}$-triviale \cite[Lemme 1.5]{CTP}. 
Ceci a \'et\'e utilis\'e dans divers travaux
r\'ecents pour infirmer la rationalit\'e stable de certaines vari\'et\'es (voir le rapport \cite{P}).

La question de la rationalit\'e (stable) des hypersurfaces cubiques lisses $X \subset \P^n_{\C}$ $(n\geq 3)$ 
a fait et continue \`a faire  l'objet de nombreux travaux. Pour tout $n\geq 3$ impair, on conna\^{\i}t depuis longtemps des familles de telles hypersurfaces dans 
$ \P^n_{\C}$    qui sont rationnelles. 
% exemples concrets : X dans P^{2n+1}, contiennent deux P^n gauches.
% exemple avec \'equation (cf. hassett rapport, model\'e sur ce qu'il dit dans P^5
% $$F_{1,2}(u_{i},v_{i}) +  G_{2,1}(u_{i},v_{i}) =0.$$
% avec F de degr\'e (1,2) en les u_{i,v_{i}} et G de degr\'e (2,1) en les u_{i},v_{i}
Certaines sont classiques, d'autres, dans $\P^5_{\C}$,  ont \'et\'e obtenues par B. Hassett
\cite{Hrat}.  Pour $n \geq 4$ pair,  on ne conna\^{\i}t aucune  hypersurface cubique lisse dans 
$ \P^n_{\C}$  
 qui soit rationnelle, ou m\^eme seulement r\'etractilement rationnelle.  On doute que toutes soient (stablement) rationnelles.

On dit qu'une vari\'et\'e irr\'eductible $X$ sur $\C$ de dimension $m$  est unirationnelle de degr\'e $d$
s'il existe une application rationnelle, dominante, de degr\'e $d$, de $\P^m_{\C} $ vers $X$.
On sait que toute hypersurface cubique lisse $X \subset \P^n_{\C}$ $(n\geq 3)$ 
est unirationnelle de degr\'e 2. Faute d'\'etablir la rationalit\'e d'une telle hypersurface $X$, on s'int\'eresse
\`a la question de l'unirationalit\'e de degr\'e impair de $X$, qui entra\^{\i}ne
que $X$ est universellement $CH_{0}$-triviale. Ceci a \'et\'e \'etabli pour certaines familles 
d'hypersurfaces cubiques dans $\P^5_{\C}$   \cite[\S 7]{HT} et \cite{H}.

 C.~Voisin \cite{V} a r\'ecemment \'etudi\'e  la question de la $CH_{0}$-trivialit\'e universelle des hypersurfaces cubiques.
 Pour $X$ de dimension 3, elle a donn\'e un crit\`ere \cite[Cor. 4.4]{V} en termes de la classe minimale sur la jacobienne interm\'ediaire.
 Cela lui permet  \cite[Thm. 4.5]{V} de montrer l'existence de telles hypersurfaces $X \subset \P^4_{\C}$, dont l'hypersurface de Fermat, qui sont 
 universellement $CH_{0}$-triviales.
Par ailleurs, elle \'etablit    \cite[Thm. 5.6]{V}  la $CH_{0}$-trivialit\'e universelle pour beaucoup d'hypersurfaces
 cubiques lisses $X \subset \P^5_{\C}$ qui sont sp\'eciales au sens de Hassett \cite{Hspec}.

\medskip

Dans cet article, je montre  de fa\c con  \'el\'ementaire : 
 {\it Pour tout entier $n \geq 4$,  toute hypersurface cubique $X$ lisse dans $\P^n_{\C}$
dont l'\'equation s'\'ecrit   
$\sum_{i}  \Phi_{i}$, o\`u les $\Phi_{i}$ sont  des formes homog\`enes \`a variables s\'epar\'ees et chacune  a au plus trois variables,
est universellement $CH_{0}$-triviale}.

\medskip

Il existe donc des hypersurfaces cubiques lisses $X$
de toute dimension $\geq 2$ qui sont universellement $CH_{0}$-triviales, ce qui
ne semblait pas connu pour  $X$ de dimension impaire au moins \'egale \`a 5.

\medskip

Je remercie Claire Voisin pour diverses suggestions.

\section{Rappels}

Le lemme suivant est un cas facile de \cite[Prop. 1.8]{CTP},
dont il r\'esulte par consid\'eration de la
 projection $X\times_{\C}Y \to Y$.

\begin{lem}\label{produit} 
Soient $X$ et  $Y$ deux vari\'et\'es connexes projectives et lisses sur $\C$.
Si $X$ est universellement $CH_{0}$-triviale, alors, pour tout corps $F$,
la projection $CH_{0}(X_{F}\times_{F} Y_{F}) \to CH_{0}(Y_{F})$ est un isomorphisme, et, pour tout
choix d'un point $P \in X(\C)$,  le morphisme $Y \to X \times Y$ donn\'e par $M \mapsto (P,M)$
induit un isomorphisme  $CH_{0}(Y_{F}) \to CH_{0}(X_{F}\times_{F} Y_{F})$.
\end{lem}

\begin{prop}\label{nihilnovumsubsole}
Soient $n\geq 2$ et $m\geq 2$ des entiers et
 $f(x_{1}, \dots, x_{n})$ et $g(y_{1}, \dots, y_{m})$ des formes homog\`enes  
 de m\^{e}me degr\'e $d$, \`a coefficients dans un corps $k$ de caract\'eristique premi\`ere \`a $d$,
 dont l'annulation d\'efinit des hypersurfaces lisses.
 Soit $X_{f} \subset \P^n_{k}$, resp. $X_{g} \subset \P^m_{k}$,
 l'hypersurface
  lisse 
   d\'efinie par  $f(x_{1}, \dots, x_{n}) + x_{0}^d=0$,
 resp. par $g(y_{1}, \dots, y_{m)} +y_{0}^d=0$. 
L'application rationnelle de $\P^n_{k} \times_{k} \P^m_{k}$ vers $\P^{n+m-1}_{k}$
envoyant le point de coordonn\'ees bihomog\`enes $(x_{0}, \dots, x_{n}; y_{0}, \dots, y_{m})$ sur le point de coordonn\'ees homog\`enes $(x_{1}/x_{0}, \dots, x_{n}/x_{0}, y_{1}/y_{0}, \dots, y_{m}/y_{0})$ induit une application rationnelle dominante de degr\'e $d$ de $X_{f} \times_{k} X_{g}$ sur l'hypersurface lisse dans $\P^{n+m-1}_{k}$ d\'efinie par l'\'equation
   $$f(z_{1}, \dots, z_{n}) - g(z_{n+1}, \dots, z_{n+m)}=0.$$
 \end{prop}

Ce fait a \'et\'e utilis\'e par Shioda et Katsura \cite{SK}
dans l'\'etude des conjectures de Hodge et de Tate pour les hypersurfaces diagonales.

 \begin{prop}\label{degre2}
 Soit $X \subset \P^n_{\C}$, $n \geq 3$, une hypersurface cubique lisse.
 
 (i)  $X$ est unirationnelle de degr\'e 2.
 
 (ii)  Pour tout corps $F$, on a $2 A_{0}(X_{F})=0$.
 \end{prop}
C'est bien connu. Pour (i), voir \cite[Thm. 2.2]{ACTP}. Pour (ii), voir  \cite[Cor. 2.3]{ACTP}.

\section{Hypersurfaces cubiques presque diagonales}

\begin{prop}\label{Satz1}
Soit $Y \subset \P^n_{\C}$, $n \geq  3$, une hypersurface cubique lisse
d'\'equation
$$f(x_{0}, \dots, x_{n-1}) + x_{n}^3=0.$$ Supposons que $Y$ est unirationnelle de degr\'e $d$.
Alors toute hypersurface cubique lisse $X\subset \P^{n+2}_{\C}$  d'\'equation
$$f(x_{0}, \dots, x_{n-1}) -  h(x_{n},x_{n+1},x_{n+2})=0$$
  est unirationnelle de degr\'e $3d$.
\end{prop}

\begin{proof}
Comme la surface cubique lisse de $\P^3_{\C}$ d'\'equation $h(y_{0},y_{1},y_{2})+y_{4}^3=0$
est, comme toute surface cubique lisse sur $\C$, rationnelle, l'\'enonc\'e r\'esulte de  la proposition \ref{nihilnovumsubsole}.
\end{proof}
%\end{proof}

Convenons qu'une forme   non singuli\`ere en $s$ variables est de type $(a_{1}, \dots,a_{r})$
 si elle est somme de $r$  formes $\Phi_{i}$ \`a variables s\'epar\'ees, que $a_{i}$
est le nombre de variables de $\Phi_{i}$, et $s = \sum_{i} a_{i}$.

\begin{theo} \label{Haupsatz1}
Soit $m \geq 1$.

(i) Toute hypersurface cubique  lisse $X$ dans $\P^{6m-1}_{\C}$ de type $(3,\dots,3)$ est unirationnelle
de degr\'e une puissance de 3.

(ii) Toute hypersurface cubique  lisse  $X$  dans $\P^{6m+1}_{\C}$ de type $(3,\dots,3,2)$
est unirationnelle de degr\'e une puissance de 3.

(iii) Toute hypersurface cubique  lisse  $X$  dans
  $\P^{6m+3}_{\C}$ de type $(3,\dots,3,1)$ est unirationnelle de degr\'e une puissance de 3.
  
 (iv) Dans chacun de ces  cas, on a $A_{0}(X_{F})=0$
pour tout corps $F$.
\end{theo}

\begin{proof} (i),  (ii)  et (iii) : C'est une cons\'equence combinatoire de la proposition \ref{Satz1}.
L'hypersurface de type $(3,1)$ est rationnelle. De $(3,1)$ et $(3,1)$ on obtient $(3,3)$.
On a donc $(3,2,1)$. De $(3,2,1)$ et $(3,1)$ on obtient $(3,3,2)$.
On a donc $(3,3,1,1)$. De $(3,3,1,1)$ et $(3,1)$ on obtient $(3,3,3,1)$.
Puis  $(3,3,3,3)$, etc. 

(iv) Compte tenu de la proposition  \ref{degre2}, le pgcd des degr\'es d'unirationalit\'e des
hypersurfaces  $X$ dans (i),  (ii) et (iii)    est \'egal \`a 1. Ceci implique   $A_{0}(X_{F})=0$
pour tout corps $F$.
\end{proof}

\begin{rema} {\rm
 On confrontera les th\'eor\`emes \ref{Haupsatz1}    et \ref{Hauptsatz2}   ci-apr\`es   avec le fait classique que toute
hypersurface cubique lisse $X \subset \P^{2n-1}_{\C}$ de type $(2,\dots,2)$ i.e.  d'\'equation $\sum_{i=1}^n f_{i}(u_{i},v_{i})=0$,
est rationnelle, car elle contient deux  espaces lin\'eaires $\P^{n-1}_{\C}$ sans point commun.}
\end{rema}

\begin{rema}\label{voisin} {\rm
Le th\'eor\`eme \ref{Haupsatz1} s'applique en particulier  aux hypersurfaces cubiques
lisses de dimension 4 donn\'ees  dans $\P^5_{\C}$ par une \'equation
$$f(x_{0},x_{1},x_{2}) + g(x_{3},x_{4},x_{5})=0.$$
Suivant une suggestion de C.~Voisin, montrons  que toute telle hypersurface est  en fait rationnelle.

Soit $S \subset \P^3_{\C} $ la surface cubique lisse d'\'equation 
$$f(x_{0},x_{1},x_{2}) - t^3=0$$
et
$T \subset \P^3_{\C}$ la surface cubique lisse d'\'equation 
$$g(x_{3},x_{4},x_{5}) - t^3=0.$$
Le plan $t=0$ d\'ecoupe une cubique lisse $C \subset S$,
et de m\^eme $D \subset T$.
Chacune des 27 droites de $S$ peut \^etre
d\'ecrite par un morphisme de $\P^1_{\C}$ vers $S$
donn\'e par un syst\`eme de formes lin\'eaires
$( a(u,t), b(u,t), c(u,t),t).$
Les 27 droites de $S$  viennent ainsi par groupes de trois droites coplanaires ayant
m\^{e}me point d'intersection avec $t=0$, qui est   un point de Eckardt de
la surface cubique $S$. Ceci d\'etermine ainsi 9 points 
de la cubique $C$.
De m\^eme chacune des
27 droites de $T$ peut \^etre
d\'ecrite par un morphisme de $\P^1_{\C}$ vers $T$
donn\'e par un syst\`eme de formes lin\'eaires
$ (d(v,t), e(v,t), h(v,t),t).$
Le produit d'une droite $l$ de $S$ et d'une droite $m$ de $T$
s'envoie par l'application rationnelle de $S\times T$ vers $X$
de la proposition \ref{nihilnovumsubsole}
sur un 2-plan $\Pi_{l,m} \subset X \subset \P^5_{\C}$ donn\'e par le morphisme de $\P^2_{\C}$ vers $\P^5_{\C}$
d\'efini par
$$( a(u,t), b(u,t), c(u,t), d(v,t), e(v,t), h(v,t) ).$$

Fixons les droites $l \subset S$ et $m \subset T$. Soit $l_1 \subset S$
une droite qui ne rencontre pas $l$. Soit $m_1 \subset T$
une droite qui rencontre $m$ en un point non situ\'e sur $t=0$.
On v\'erifie alors que les deux plans  $\Pi_{l,m} \subset X \subset \P^5_{\C}$
et  $\Pi_{l_1,m_1} \subset X \subset \P^5_{\C}$ ne se rencontrent pas.
Ceci implique que l'hypersurface cubique {\it $X \subset \P^5_{\C}$ est rationnelle}.}
  \end{rema}

\'Etablissons   un \'enonc\'e d'int\'er\^et g\'en\'eral.

\bpr\label{courbe}
Soit $X/\C$ une vari\'et\'e   projective et lisse telle que $A_{0}(X)=0$.
S'il existe une courbe $\Gamma$ projective,  lisse, connexe et un morphisme $\Gamma \to X$
tels que, pour tout corps $F$, l'application induite
$CH_{0}(\Gamma_{F}) \to CH_{0}(X_{F})$ soit surjective, alors, pour tout corps $F$, l'application
 ${\rm deg}_{F} :CH_{0}(X_{F}) \to \Z$ est un isomorphisme, en d'autres termes $X$ est universellement 
$CH_{0}$-triviale.
\epr

 \begin{proof}  
 %L'hypoth\`ese $A_{0}(X)=0$ permet de produire un entier $N>0$
 % tel  que $NA_{0}(X_{F})=0$ pour tout corps $F$  (cf. \cite[Prop. 11]{CT}).
Soit $J$ la jacobienne de $\Gamma$. Pour tout corps $F$, on a
  $A_{0}(\Gamma_{F})=J(F)$.
Notons $K=\C(X)$ le corps des fonctions de $X$.
L'hypoth\`ese $A_{0}(X)=0$ 
implique que
 la vari\'et\'e d'Albanese de 
$X$ est triviale, comme on peut voir de diverses fa\c cons.
Par exemple,  la combinaison de 
 \cite[Lemma 1.7]{ACTP} et des r\'esultats rassembl\'es dans \cite[Prop. 1.8]{ACTP}
 donne $H^1(X,{\mathcal O}_{X})=0$.
Un point de $J(\C(X))$ d\'efinit une application rationnelle
de $X$ dans $J$,  donc un morphisme de $X$ dans $J$ car une application rationnelle d'une vari\'et\'e lisse dans une vari\'et\'e ab\'elienne est partout d\'efinie.
Mais comme la vari\'et\'e d'Albanese de $X$ est triviale, tout tel morphisme est constant. On a donc $J(\C)=J(\C(X))$.
Ainsi l'image de $A_{0}(\Gamma_{K})$ dans $A_{0}(X_{K})$
est dans l'image de l'application compos\'ee
$$J(\C)= A_{0}(\Gamma) \to A_{0}(X)  \to A_{0}(X_{K}),$$
et donc est nulle car $A_{0}(X)=0$.
% simplification (Cao Yang, 7 octobre)
%Ainsi $J(K)/N=J(\C)/N= 0$ et donc $A_{0}(X_{K})=NA_{0}(X_{K})=0$. 
Par hypoth\`ese, l'application  $A_{0}(\Gamma_{K}) \to A_{0}(X_{K})$ est surjective. On  a donc
 $ A_{0}(X_{\C(X)})=A_{0}(X_{K})=0$.
Ainsi  (\cite[Lemma 1.3]{ACTP}) la diagonale se d\'ecompose dans $X \times_{\C} X$ et,
  pour tout corps $F$,
on a $A_{0}(X_{F})=0$.
\end{proof}

 \begin{rema}{\rm
 Soit $X$ une $\C$-vari\'et\'e projective et lisse, g\'eom\'etriquement connexe.
Soit $\delta(X)$ le plus petit entier $\delta \leq {\rm dim}(X)$ pour lequel il existe
une $\C$-vari\'et\'e  projective  lisse connexe $Y$ de dimension $\delta $ et un  morphisme $\phi : Y \to X$
tels que, pour tout corps $F$,
l'application
$\phi_{F,*} : CH_{0}(Y_{F}) \to CH_{0}(X_{F})$ soit surjective.
Pour $X$  hypersurface cubique lisse de dimension $\geq 3$ {\it tr\`es g\'en\'erale},   C.~Voisin  \cite[Thm. 5.3]{V} montre  que
$\delta(X)   < {\rm dim}(X) $ implique $\delta(X) =0$. La proposition \ref{courbe} peut se reformuler ainsi :
{\it Pour toute vari\'et\'e $X$ projective et lisse telle que $A_{0}(X)=0$,
$\delta(X) \leq 1$  implique $\delta(X)=0$.}}
\end{rema}

\begin{prop}\label{Satz2}
 Soit $Y \subset \P^n_{\C}$ avec $n \geq  3$ une hypersurface cubique lisse
d'\'equation
$$f(x_{0}, \dots, x_{n-1}) + x_{n}^3=0.$$ Supposons que $Y$ est
universellement $CH_{0}$-triviale.  Alors :

(i) Toute hypersurface cubique  lisse $X \subset \P^{n+1}_{\C}$
d'\'equation
$$f(x_{0}, \dots, x_{n-1}) - g(x_{n},x_{n+1})=0$$ 
%dans $\P^{n+1}_{\C}$
est universellement $CH_{0}$-triviale.

(ii) Toute hypersurface cubique lisse $X  \subset \P^{n+2}_{\C}$ d'\'equation
$$f(x_{0}, \dots, x_{n-1}) -  h(x_{n},x_{n+1},x_{n+2})=0$$
 est universellement $CH_{0}$-triviale.
\end{prop}

 \begin{proof}
(i) 
Soit $\Gamma \subset \P^2_{\C}$ la courbe d'\'equation  
$g(u,v) + w^3=0$. La proposition \ref{nihilnovumsubsole} fournit une application
rationnelle dominante, de degr\'e 3,  de $Y  \times_{\C} \Gamma$ vers $X$.
Il existe une vari\'et\'e $W$ projective et lisse, un morphisme birationnel
$W \to Y \times_{\C} \Gamma$ induisant un morphisme $f : W \to X$ dominant g\'en\'eriquement fini
de degr\'e 3.
 Il existe un point $P \in Y(\C)$ tel que le morphisme 
$\Gamma  \to Y \times_{\C} \Gamma$ donn\'e par $M \mapsto (P,M)$ se rel\`eve en
un morphisme $\Gamma \to W$. Soit $F$ un corps contenant $\C$.  
Consid\'erons les applications induites sur les groupes de Chow des z\'ero-cycles
au-dessus de $F$:  $$CH_{0}(\Gamma_{F}) \to CH_{0}(W_{F})
\to CH_{0}(Y_{F} \times_{F} \Gamma_{F}).$$
  La seconde application  est un isomorphisme par invariance birationnelle du groupe de Chow des
z\'ero-cycles sur les vari\'et\'es projectives et lisses. L'application compos\'ee est un isomorphisme
par application du lemme \ref{produit}.  Ainsi l'application $CH_{0}(\Gamma_{F}) \to 
CH_{0}(W_{F})$ est un isomorphisme. 
Comme le morphisme $ W \to X$ est g\'en\'eriquement fini de degr\'e 3, 
le groupe $3 A_{0}(X_{F}) \subset A_{0}(X_{F}) $ est dans l'image de $A_{0} (W_{F})$, et donc dans l'image
de $A_{0}(\Gamma_{F}) \to A_{0}(X_{F})$ via l'application compos\'ee
$\Gamma \to W \to X$. D'apr\`es la proposition \ref{degre2}, on a $2 A_{0}(X_{F})=0$.
Ainsi le groupe  $A_{0}(X_{F})$ est dans l'image de $A_{0}(\Gamma_{F})$.
La proposition \ref{courbe} donne alors $A_{0}(X_{F})=0$.

(ii) La d\'emonstration est   plus simple. La courbe $\Gamma$ est remplac\'ee par
la surface cubique lisse $S$ d'\'equation $g(u,v,t) + w^3=0$, qui comme toute surface cubique lisse est une surface rationnelle,
et donc satisfait $A_{0}(S_{F})=0$ pour tout corps $F$.
\end{proof}

\begin{theo}\label{Hauptsatz2}
Toute hypersurface cubique lisse $X \subset \P^n_{\C}$ de dimension au moins 2
 dont l'\'equation
est donn\'ee par une forme   $\sum_{i}  
\Phi_{i}$, o\`u les $\Phi_{i}$ sont \`a variables
 s\'epar\'ees et chacune  a au plus 3 variables,
est universellement $CH_{0}$-triviale.
\end{theo}
\begin{proof}
Une forme de type $(1,1,1,1)$ d\'efinit une surface cubique lisse, donc rationnelle, donc
universellement $CH_{0}$-triviale.
En appliquant de fa\c con r\'ep\'et\'ee le th\'eor\`eme \ref{Satz2}(i), on obtient  que toute forme  non singuli\`ere diagonale, i.e. de type 
$(1,\dots,1)$ en au moins 4 variables, 
d\'efinit une hypersurface lisse universellement $CH_{0}$-triviale. 
En appliquant le th\'eor\`eme \ref{Satz2} de fa\c con r\'ep\'et\'ee,  on peut  ensuite remplacer chacun des $1$ dans $(1,\dots,1)$ par $2$ ou par $3$.
\end{proof}

\begin{rema}  {\rm
   C'est une question ouverte s'il existe une  hypersurface cubique lisse
 $X \subset  \P^4_{\C}$ avec une application rationnelle $\P^3_{\C} \dots \to X$ de degr\'e impair.
 La d\'emonstration de la proposition \ref{Satz2}(i), via la proposition \ref{courbe},   m\`ene \`a la question a priori plus facile :
Pour $X \subset  \P^4_{\C}$ hypersurface cubique lisse quelconque, existe-t-il
une application rationnelle de degr\'e impair $\Gamma \times_{\C} \P^2 \dots \to X$, avec $\Gamma$
courbe convenable ?  }
\end{rema}

\end{document}